\documentclass[11pt]{amsart}
\usepackage{amssymb,amsthm,amsmath,amstext}
\usepackage{mathrsfs}  
\usepackage{mathdots}
\usepackage{xcolor}
\usepackage{multirow}
\usepackage{enumitem}
\usepackage{cases}
\usepackage{graphicx}
\usepackage{bm}
\usepackage{mathtools}
\usepackage{tikz}
\usepackage{tikz-cd}
\tikzset{every picture/.style={line width=0.75pt}} 
\usepackage{amsmath}
\usepackage{tikz}
\usepackage{mathdots}
\usepackage{yhmath}
\usepackage{cancel}
\usepackage{color}
\usepackage{siunitx}
\usepackage{array}
\usepackage{multirow}
\usepackage{amssymb}
\usepackage{tabularx}
\usepackage{extarrows}
\usepackage{booktabs}
\usetikzlibrary{fadings}
\usetikzlibrary{patterns}
\usetikzlibrary{shadows.blur}
\usetikzlibrary{shapes}

\usepackage{subfig}

\usepackage[left=1.0in,top=0.95in,right=1.0in,bottom=0.95in]{geometry}

\usepackage[all]{xy}

\newcommand{\Z}{\mathbb{Z}}

\newcommand{\PP}{\mathbb{P}}

\newcommand{\M}{\mathcal{M}}
\newcommand{\G}{\mathcal{G}}
\newcommand{\W}{\mathcal{W}}

\newcommand{\Mbar}{\overline{\mathcal{M}}}
\newcommand{\calR}{\mathcal{R}}
\newcommand{\calO}{\mathcal{O}}

\newcommand{\g}[2]{g^{#1}_{#2}}
\newcommand{\BN}[3]{\mathcal{M}^{#2}_{#1,#3}}

\DeclareMathOperator{\Pic}{Pic}
\DeclareMathOperator{\codim}{codim}
\newcommand{\floor}[1]{\left\lfloor #1 \right\rfloor}
\newcommand{\ceil}[1]{\left\lceil #1 \right\rceil}
\newcommand{\dmax}{d_{max}}

\usepackage[]{hyperref}
\hypersetup{pdftitle={Max BN loci via degenerations and 2 cov},pdfauthor={Andrei Bud and Richard Haburcak}} 
\hypersetup{colorlinks=true,linkcolor=blue,anchorcolor=blue,citecolor=blue}

\usepackage[nameinlink]{cleveref}

\newtheorem{theorem}{Theorem}[section]
\newtheorem{lemma}[theorem]{Lemma}
\newtheorem{prop}[theorem]{Proposition}
\newtheorem{cor}[theorem]{Corollary}

\theoremstyle{definition}

\theoremstyle{definition}
\newtheorem{remark}[theorem]{Remark}

\newtheorem{theoremintro}{Theorem}
\newtheorem{conjectureintro}{Conjecture}

\title[Maximal Brill--Noether loci via degen. and double covers]{Maximal Brill--Noether loci via degenerations and double covers}

\author{Andrei Bud}
 \address{Goethe Universit\"at Frankfurt am Main, Institut f\"ur Mathematik, Robert-Mayer Strasse 6-8}
\email{andreibud95@protonmail.com}

\author{Richard Haburcak}
\address{Department of Mathematics\\%
	The Ohio State University\\%
	100 Math Tower\\%
	Columbus, OH 43210}
\email{haburcak.1@osu.edu}

\begin{document}
\thispagestyle{empty}
\vspace*{-.25cm}
\begin{abstract}
    Using limit linear series on chains of curves, we show that closures of certain Brill--Noether loci contain a product of pointed Brill--Noether loci of small codimension. As a result, we obtain new non-containments of Brill--Noether loci, in particular that all dimensionally expected non-containments hold for expected maximal Brill--Noether loci. Using these degenerations, we also give a new proof that Brill--Noether loci with expected codimension $-\rho\leq \lceil g/2\rceil$ have a component of the expected dimension. Additionally, we obtain new non-containments of Brill--Noether loci by considering the locus of the source curves of unramified double covers.\\

    \vspace{-2em}
\end{abstract}

\maketitle

\section*{Introduction}

The main
theorem of classical Brill--Noether theory \cite{gieseker,griffiths_harris}
shows that if $C$ is a general smooth projective curve of genus $g$, then $C$ admits a
nondegenerate (not lying in a hyperplane) map $C \to \PP^r$ of degree $d$ if and only if the
\emph{Brill--Noether number}
\[
\rho(g,r,d) \coloneqq g - (r+1)(g-d+r) \geq 0.
\]
A nondegenerate degree $d$ map $C \to \PP^r$ corresponds to a line bundle $L\in\Pic(C)$ of degree $d$ and a subspace $V \subseteq H^0(C,L)$ of dimension $r + 1$. The pair $(L, V)$ is called a linear system of degree $d$ and dimension $r$ on $C$, or a $g^r_d$ on $C$ for short.

In the last few years, there has been a renewed focus on \emph{refined Brill--Noether theory}, which aims to understand linear systems on a curve in a component of a Brill--Noether locus \[\BN{g}{r}{d} = \{ C \in \M_g \;\mid\; C \text{~admits~a~} \g{r}{d} \}\] when $\rho(g,r,d)<0$. In particular, there have been major advances in a refined Brill--Noether theory for curves of fixed gonality \cite{cook-powell_jensen,jensen_ranganathan,larson_larson_vogt_2020global,larson_refined_BN_Hurwitz,pflueger}. Relatively little is known about the geometry of Brill--Noether loci in general. It is known that $\BN{g}{r}{d}$ is a proper subvariety of $\M_g$, which can potentially have multiple components and satisfies $\codim \BN{g}{r}{d} \leq \max\{0,-\rho(g,r,d)\}$, see \cite{steffen_1998}, where $-\rho(g,r,d)$ is the \emph{expected codimension}. See \Cref{subsec:background BN loci} for more details.

By adding basepoints and subtracting non-basepoints, one obtains many trivial containments of Brill--Noether loci. The \emph{expected maximal Brill--Noether loci} are precisely the loci which do not admit such trivial containments, for a detailed characterization see \Cref{subsec:background exp max BN loci}. Inspired by work on lifting line bundles on K3 surfaces, Auel and the second author posed a conjecture in \cite{auel_haburcak_2022} concerning potential containments of the ``largest" Brill--Noether loci.

\begin{conjectureintro}[Maximal Brill--Noether Loci Conjecture]
\label{Conj Max BN loci} 
For any $g \geq 3$, except for $g=7,8,9$, the expected maximal
Brill--Noether loci are maximal with respect to containment.
\end{conjectureintro}
\noindent

There has been a flurry of recent progress on this conjecture in work of Auel--Haburcak--Larson, Bud, and Teixidor i Bigas \cite{ahl_BN_via_gonality_2023,bud2024brillnoether,bigas2023brillnoether}. In particular, \Cref{Conj Max BN loci} holds in genus $g\leq 23$ and by work of Choi, Kim, and Kim \cite{CHOI2022,CHOI20141458}, in genus $g$ such that \[g+1 \text{ or } g+2\in\ \{\operatorname{lcm}(1,2,\dots,n) \text{ for some } n\in\mathbb{N}_{\ge 3} \}.\]

In this paper, we give new non-containments of Brill--Noether loci. One expects that a Brill--Noether locus of large expected dimension is not contained in a Brill--Noether locus of small expected dimension. We prove that this is indeed the case. 

\begin{theoremintro} \label{theoremintro rho}
    Let $\BN{g}{r}{d}$ and $\BN{g}{s}{e}$ be expected maximal Brill--Noether loci. If $\rho(g,s,e)<\rho(g,r,d)$, then $\BN{g}{r}{d} \nsubseteq \BN{g}{s}{e}$.
\end{theoremintro}
We show that given an expected maximal Brill--Noether locus $\BN{g}{r}{d}$, we can find a curve in the closure of $\BN{g}{r}{d}$ in $\Mbar_g$ that is not contained in the closure of any other expected maximal Brill--Noether locus $\BN{g}{s}{e}$ with $\rho(g,s,e)<\rho(g,r,d)$. To do this, we use limit linear series to show that the closure of $\BN{g}{r}{d}$ contains a product of Brill--Noether loci with prescribed ramification having expected codimension $1$ or $2$. Then Brill--Noether additivity and a few base cases yield \Cref{theoremintro rho}.

Furthermore, we give a new proof of the existence of a component of a Brill--Noether locus of the expected dimension.

\begin{theoremintro}\label{theoremintro dimension}
        If $d\leq 2g-2$ and $-\rho(g,r,d)\leq \ceil{g/2}$,
then $\BN{g}{r}{d}$ has a component of the expected dimension.
\end{theoremintro}

We note that this does not improve the currently best known results on the existence of components of the expected dimension, which are given in \cite{pflueger_legos,bigas2023brillnoether}. However, our method has the advantage of avoiding many of the combinatorial intricacies appearing in the previous proofs. 

We also study non-containments of Brill--Noether loci coming from restrictions on linear series on a curve $\widetilde{C}$ admitting an \'{e}tale double cover $ \widetilde{C}\to C$ of a curve of genus $g$. In particular, the image, $\operatorname{Im}(\chi_g)$, of the map $\chi_g: \calR_g \to \M_{2g-1}$ sending the double cover to the source curve interacts interestingly with the Brill--Noether stratification of $\M_{2g-1}$. For double covers, Bertram shows in \cite[Theorem 1.4]{Bertram87} that $\operatorname{Im}(\chi_g)$ is contained in certain Brill--Noether loci. Conversely, Schwarz shows in \cite[Theorem 1.1]{SchwarzPrym} that for a general double cover $\widetilde{C}\to C$, letting $\tilde{g}$ be the genus of $\widetilde{C}$, if $\rho(\tilde{g},r,d)<-r$, then $\widetilde{C}$ admits no $\g{r}{d}$. Using these restrictions, as well as ideas of Aprodu and Farkas \cite{FarkasAprodu-Greenconj}, we show infinitely many non-containments of expected maximal Brill--Noether loci.

\begin{theoremintro}\label{theoreminto Prym}
    Let $g = 1 + r(r+1) + 2\varepsilon$ for some $0\leq \varepsilon < \frac{r}{2}$ and let $s,d$ be positive integers satisfying either 
    \begin{itemize}
        \item $\rho(g,s,d) = -s-1$, or 
        \item $\rho(g,s,d) = -s$, $d$ is odd and $s \not\equiv 3 \pmod{4}$.
    \end{itemize}
    Then there is a non-containment 
    \[ \mathcal{M}^r_{g,g-1} \nsubseteq \mathcal{M}^s_{g,d}. \]
\end{theoremintro}

Already taking $\varepsilon=0$ gives infinitely many non-containments of expected maximal Brill--Noether loci which are not implied by \Cref{theoremintro rho}, see \Cref{cor: eps=0 noncontainments}.

\subsection*{Outline} In \Cref{sec: background}, we recall facts about Brill--Noether loci, limit linear series, and Prym curves. In particular, we give more precise definitions of expected maximal Brill--Noether loci in \Cref{subsec:background exp max BN loci}, including some useful facts for our proofs. In \Cref{sec: non-cont for small rho}, we prove non-containments of pointed Brill--Noether loci of small codimension which act as the base cases for our proof of \Cref{theoremintro rho}. In \Cref{sec: dim exp noncont}, we prove our main technical result, \Cref{Prop BN loci split into BN divisors and codim 2} and give a proof of \Cref{theoremintro rho} as \Cref{Thm:noncontainment based on rho}. In \Cref{sec: comp of exp dim}, we use an inductive argument and the argument of \Cref{Prop BN loci split into BN divisors and codim 2} to prove \Cref{theoremintro dimension}. Finally, in \Cref{sec: noncont from prym}, we prove additional non-containments of Brill--Noether loci coming from Prym curves.

\subsection*{Acknowledgements} We are grateful to Gavril Farkas, Hannah Larson, and Martin
M\"oller for helpful discussions and comments on this paper, and to the organizers of ``Vector bundles and combinatorial algebraic geometry" at Goethe Universit\"at Frankfurt am Main, where this project was started. We also thank the anonymous referee for their helpful comments and feedback. The first author acknowledges support by Deutsche Forschungsgemeinschaft (DFG, German Research Foundation) through the Collaborative Research Centre TRR 326 Geometry and Arithmetic of Uniformized Structures, project number 444845124. The second author would like to thank the Hausdorff Research Institute for Mathematics funded by the Deutsche Forschungsgemeinschaft (DFG, German Research Foundation) under Germany's Excellence Strategy – EXC-2047/1 – 390685813, and the Max--Planck Institut f\"ur Mathematik Bonn for their hospitality and financial support.

\section{Background}\label{sec: background}

\subsection{Brill--Noether loci}\label{subsec:background BN loci}
Brill--Noether theory studies how curves map to projective space. A map $C\to \PP^s$ factors as a non-degenerate map $C\to \PP^r$ and the linear embedding $\PP^r\subseteq \PP^s$. We restrict our attention to non-degenerate maps $C\to \PP^r$, which are determined by a $\g{r}{d}$, that is, an element of \[G^r_d(C) \coloneqq \{(L,V) \;\mid\; L\in\Pic^d(C),~ V\subseteq
H^0(C,L),~ \dim V=r+1\}.\] There is a natural globalization of $G^r_d(C)$ to a moduli space $\G^r_{g,d}$ over the moduli space $\M_g$ of smooth curves of genus $g$, where the natural map $\G^r_{g,d}\to\M_g$ has fiber $G^r_d(C)$ above $C$. The Brill--Noether loci
\[
\BN{g}{r}{d} \coloneqq \{C\in\M_g \;\mid\; C \text{ admits a
$\g{r}{d}$}\}
\] 
are the images of the corresponding maps $\G^r_{g,d}\to\M_g$.

Many classical theorems in Brill--Noether theory can be restated in terms of components of $\G^r_{g,d}$. For example, the classical Brill--Noether theorem states that $\G^r_{g,d}$ has a unique component surjecting onto $\M_g$ when $\rho(g,r,d)\geq 0$, and this component has relative dimension $\rho(g,r,d)$ \cite{pflueger_legos}. The expected relative dimension of $\G^r_{g,d}$ is $\rho(g,r,d)$, in particular when $\rho(g,r,d)<0$, $\BN{g}{r}{d}$ has expected codimension $-\rho(g,r,d)$ in $\M_g$. 

When Brill--Noether loci are equidimensional, perhaps even irreducible, one can use simple dimension arguments to prove non-containments of Brill--Noether loci, large loci cannot be contained in small loci. However, only Brill--Noether loci with $\rho=-1$ and $\BN{g}{2}{d}$ with $\rho=-2$ are known to be irreducible \cite{CHOI2022,EisenbudHarrisBN-1,steffen_1998}. More is known about the existence of components of expected dimension, however not much is known about equidimensionality of $\BN{g}{r}{d}$. It is known that the codimension of any component of $\BN{g}{r}{d}$ is at most $-\rho(g,r,d)$, and when $-3 \le \rho(g,r,d)\le -1$ (additionally assuming $g\geq 12$ when $\rho(g,r,d)=-3$), the Brill--Noether loci are equidimensional of the expected dimension \cite{EdidinThesis,steffen_1998}. Complicating the picture, components of larger than expected dimension can exist, examples include Castelnuovo curves, see for example \cite[Remark 1.4]{pflueger_legos}.

When $\rho$ is not too negative, avoiding the Castelnuovo curve examples, it is expected that there is a component of expected dimension. Recently, Pflueger and Teixidor i Bigas independently showed that when $\rho\geq -g+3$, $\BN{g}{r}{d}$ has a component of expected dimension \cite{pflueger_legos,bigas2023brillnoether}. We give a new proof of the existence of a component of expected dimension for Brill--Noether loci of expected codimension $\leq \ceil{g/2}$.

\subsection{Expected maximal Brill--Noether loci}\label{subsec:background exp max BN loci}
Many statements in refined Brill--Noether theory can be restated as studying the stratification of $\M_g$ by Brill--Noether loci. There are \emph{trivial containments}
$\BN{g}{r}{d}\subseteq \BN{g}{r}{d+1}$ obtained by adding a basepoint
to a $\g{r}{d}$ on $C$; and $\BN{g}{r}{d}\subseteq \BN{g}{r-1}{d-1}$
when $\rho(g,r-1,d-1)<0$ by subtracting a non-basepoint \cite{Farkas2000,Lelli-Chiesa_the_gieseker_petri_divisor_g_le_13}. The \emph{expected maximal Brill--Noether loci} are defined as the Brill--Noether loci not admitting these trivial containments. Concretely, for fixed $r\geq 1$ a Brill--Noether locus $\BN{g}{r}{d}$ is expected maximal if $d$ is maximal such that $\rho(g,r,d)<0$ and $\rho(g,r-1,d-1)\geq 0$. Accounting for Serre duality, which shows $\BN{g}{r}{d}=\BN{g}{g-d+r-1}{2g-2-d}$, every Brill--Noether locus is contained in at least one expected maximal Brill--Noether locus. As observed in \cite[Lemma 1.1]{ahl_BN_via_gonality_2023}, the expected maximal Brill--Noether loci are exactly the $\BN{g}{r}{d}$ such that $2r\leq d\leq g-1$ where $r$ satisfies
\begin{equation} \label{rcases} 
1\leq r \leq 
\begin{cases}
	\ceil{\sqrt{g}-1} & \text{if~} g\geq\floor{\sqrt{g}}^2+\floor{\sqrt{g}}\\
	\floor{\sqrt{g}-1} & \text{if~} g<\floor{\sqrt{g}}^2+\floor{\sqrt{g}},
\end{cases}
\end{equation} 
and for each such $r$
\begin{equation} \label{dmaxdef}
d = \dmax(g, r) \coloneqq  r+\left\lceil \frac{gr}{r+1} \right\rceil -1.
\end{equation}

In \cite{auel_haburcak_2022}, Auel and the second author posed \Cref{Conj Max BN loci}, which
says that the expected maximal Brill--Noether loci should be maximal
with respect to containment, except when $g = 7, 8, 9$. Concretely, for any two $\BN{g}{r}{d}$ and $\BN{g}{s}{e}$ expected maximal, there should exist a curve $C$ admitting a $\g{r}{d}$ but no $\g{s}{e}$. We note that the exceptional cases in genus
$7, 8,$ and $9$, come from unexpected containments of Brill--Noether
loci obtained from projections from points of multiplicity $\ge 2$ in
genus $7$ and $9$ \cite[Propositions 6.2 and 6.4]{auel_haburcak_2022}
or from a trisecant line in genus $8$, as shown by Mukai \cite[Lemma
3.8]{Mukai_Curves_and_grassmannians_1993}.)  Following this, they proved \Cref{Conj Max BN loci} in genus $g\leq 19$, $22$, and $23$ using various K3 surface techniques and Brill--Noether theory for curves of fixed gonality. Moreover, work of Choi, Kim, and Kim \cite{CHOI2022,CHOI20141458} showing that Brill--Noether loci with $\rho=-1,-2$ are distinct verifies \Cref{Conj Max BN loci} in infinitely many genera, cf. \cite{auel_haburcak_2022}. More recently, Auel--Haburcak--Larson employed the gonality stratification and the refined Brill--Noether theory for curves of fixed gonality to verify the $g=20$ case \cite{ahl_BN_via_gonality_2023}, and the first author has verified the $g=21$ case by employing a degeneration argument and studying strata of differentials \cite{bud2024brillnoether}. Various non-containments of expected maximal Brill--Noether loci are also known, for details see \cite{auel_haburcak_2022,ahl_BN_via_gonality_2023,bud2024brillnoether,bigas2023brillnoether}. 

We end with a few useful facts about expected maximal Brill--Noether loci.

\begin{lemma}[{\cite[Lemma 4.1]{ahl_BN_via_gonality_2023}}]\label{lem: exp max rho formula}
    Let $g\mod r+1$ be the smallest non-negative representative. For an expected maximal Brill--Noether locus $\BN{g}{r}{d}$, we have $-\rho(g,r,d)=r+1-(g\mod r+1)$.
\end{lemma}

Moreover, for $r$ satisfying \Cref{rcases}, the expected maximal Brill--Noether loci are exactly the Brill--Noether loci with the largest expected dimension.

\begin{lemma}\label{lem: small -rho means exp max}
    For $2r\leq d\leq g-1$ and $r$ satisfying \Cref{rcases} if $-r-1\leq \rho(g,r,d)\leq -1$, then $\BN{g}{r}{d}$ is expected maximal.
\end{lemma}
\begin{proof}
    A straightforward computation shows that if $-r-1\leq \rho(g,r,d)\leq -1$, then $d\geq \dmax(g,r)$ and $\rho(g,r,d+1)= \rho(g,r,d)+r+1\geq 0$. For $r$ satisfying \Cref{rcases} and $\rho(g,r,d)<0$, we have $r+1\leq g-d+r$, hence $\rho(g,r-1,d-1)=\rho(g,r,d)+g-d+r\geq 0$. Thus $\BN{g}{r}{d}$ is expected maximal.
\end{proof}

\subsection{Limit linear series and pointed Brill--Noether loci}\label{subsec: limit lin background}
We recall the basics of limit linear series and pointed Brill--Noether loci. Let $C$ be a smooth curve. We follow the standard terminology from \cite{basiclimitlinear} and \cite{Farkas2000}.

Let $g,r,d$ be positive integers satisfying $d<g+r$. Given a curve $C$ of genus $g$, a linear series $\ell=(L,V)\in G^r_d(C)$, and fixing a point $p\in C$, we order the finite set $\{\operatorname{ord}_p(\sigma)\}_{\sigma \in  V}$ of vanishing orders of sections, giving a \emph{vanishing sequence} \[a^{\ell} (p): 0\leq a_0^{\ell}(p) < a_1^{\ell}(p)\cdots < a_r^{\ell}(p) \leq d\] of non-negative integers. The \emph{ramification sequence} of $\ell$ at $p$ \[0\leq b_0^{\ell}(p)\leq \cdots \leq b_r^{\ell}(p) \leq d-r\] is given by $b_i^{\ell}(p) \coloneqq a_i^{\ell}(p)-i$, and the \emph{weight} of $\ell$ at $p$ is \[w^{\ell}(p)=\sum_{i=1}^r b_i^{\ell}(p).\] When the linear series $\ell$ is understood, we omit it from the notation.

We call a sequence of integers $0\leq b_0 \leq \cdots b_r \leq d-r$ a \emph{ramification sequence of type $(r,d)$} and weight $w(b)=\sum b_i$, and given two ramification sequences of type $(r,d)$, we say $(b_i)\leq (c_i)$ when $b_i\leq c_i$ for all $0\leq i \leq r$. Similarly, we call a sequence of integers $0\leq a_0<a_1<\cdots < a_r\leq d$ a \emph{vanishing sequence of type $(r,d)$}. Given $n$ smooth points $p_1,\dots,p_n$ on a curve $C$ and $n$ ramification sequences $b^1,\dots,b^n$ of type $(r,d)$, we define \[G^r_d(C,(p_1,b^1),\dots,(p_n,b^n))\coloneqq\{\ell\in G^r_d(C) \;\mid\; b^{\ell}(p_i)\geq b^i\},\] which is a determinantal variety of expected dimension \[\rho(g,r,d,b^1,\dots,b^n)\coloneqq \rho(g,r,d)-\sum_{i=1}^{n} w(b^i),\] which is the \emph{adjusted Brill--Noether number}. If the linear series $\ell$ and the vanishing sequences are understood, we sometimes abbreviate $\rho(g,r,d,b^1,\dots,b^n)=\rho(\ell,p_1,\dots,p_n)$ to emphasize the points rather than the ramification sequence.

We will work mainly with vanishing sequences, hence given a ramification sequence $(b_i)$ of type $(r,d)$ we define the \emph{associated vanishing sequence} as $(a_i)\coloneqq(b_i+i)$.

Similarly, one can define pointed versions of $W^r_d(C)$, namely \begin{align*}
    W^r_d(C,(p_1,b^1),\dots,(p_n,b^n))\coloneqq \{L\in \Pic^d(C) \;\mid\; & h^0(C,L(-a_i^j p_j))\geq r+1-i\\ & \text{for all } 0\leq i \leq r \text{ and all }1\leq j \leq n\}.
\end{align*}

One may also globalize these constructions, as with $\W^r_d$ and $\G^r_{g,d}$. Namely, given ramification sequences $b^1,\dots,b^n$ of type $(r,d)$, with $a^1,\dots,a^n$ the associated vanishing sequences, we define the \emph{pointed Brill--Noether loci} \[\BN{g}{r}{d}(a^1,\dots,a^n)\coloneqq \{C\in \M_{g,n} \;\mid\; G^r_d(C,(p_1,b^1),\dots,(p_n,b^n))\neq \emptyset\}\subseteq \M_{g,n}.\] When the entries of the vanishing sequences are consecutive positive numbers, the corresponding point is simply a base-point of the linear series. In particular, by subtracting the base-point $a_0 p$, one sees that $\BN{g}{r}{d}(a_0,a_0+1,\dots,a_0+r)=\BN{g}{r}{d-a_{_0}}$, viewed in $\mathcal{M}_{g,1}$ as the preimage of $\BN{g}{r}{d-a_0}$ under the forgetful map $\M_{g,1}\to \M_g$.

For a curve $C$ of compact type (i.e. every node of $C$ is disconnecting, or equivalently a curve whose dual graph is a tree or whose Jacobian is compact), a \emph{crude limit $\g{r}{d}$} on $C$ is a collection of ordinary linear series 
\[\ell=\{\ell_Y=(L_Y,V_Y)\in G^r_d(Y) \;\mid\; Y\subseteq C \text{ is an irreducible component}\}\]
satisfying a compatibility condition on the intersections of components. Namely, if $Y$ and $Z$ are irreducible components of $C$ with $p=Y\cap Z$, then \[a^{\ell_Y}_i(p) + a^{\ell_Z}_{r-i}(p) \geq d \text{ for all $0\leq i \leq r$.}\] When equality holds for each $i$, we say that $\ell$ is a \emph{refined limit $\g{r}{d}$}. The linear series $\ell_Y\in G^r_d(Y)$ is called the \emph{$Y$-aspect} of the limit linear series $\ell$.

In \cite[Lemma 3.6]{basiclimitlinear}, it is proven that the adjusted Brill--Noether number is additive. Namely \[\rho(g,r,d)\geq \sum_{Y\subseteq C} \rho(\ell_Y,b^{\ell_Y}(p_1),\dots,b^{\ell_Y}(p_k)),\] where $p_1,\dots,p_k$ are the intersections of $Y$ with the other components of $C$, and equality holds exactly when $\ell$ is a refined limit linear series. Furthermore, due to the determinantal nature of $G^r_d(C,(p_1,b^1),\dots,(p_n,b^n))$, as shown in \cite[Corollary 3.5]{basiclimitlinear}, limit linear series that move in a space of the expected dimension smooth to nearby curves.

\subsection{Prym--Brill--Noether loci}\label{Prym background}

We recall some basic facts about the Prym moduli space $\calR_g$ of unramified double covers of curves of genus $g$, and Prym--Brill--Noether loci which are useful in \Cref{sec: noncont from prym}.

Recall that the moduli space of Prym curves
\[\calR_g\coloneqq \{[C,\eta] \;\mid\; C\in \M_g,~\eta\in\Pic^0(C)\setminus\{\calO_C\},~\eta^{\otimes2}\cong\calO_C\},\] 
introduced by Mumford in his seminal paper \cite{Mumford_prym} and further popularized by Beauville in \cite{Beauville_prym}, parameterizes smooth curves of genus $g$ together with a $2$-torsion point of the Jacobian of $C$. The data of such a pair $[C,\eta]\in\calR_g$ is equivalent to the datum of an unramified double cover $f:\widetilde{C}\to C$ where $\widetilde{C}\coloneqq\operatorname{Spec}(\calO_C\oplus \eta)$. As the cover is unramified, we immediately see that the genus of $\widetilde{C}$ is given by $g(\widetilde{C}) = 2g(C)-1=2g-1$. The \'{e}tale double cover $f:\widetilde{C}\to C$ induces a norm map \[\operatorname{Nm}_f:\Pic^{2g-2}\left(\widetilde{C} \right)\to \Pic^{2g-2}(C),~ \operatorname{Nm}_f \left( \calO_{\widetilde{C}}(D)\right) \coloneqq \calO_C\left(f(D)\right).\] The Prym moduli space $\calR_{g}$ parametrizing unramified double covers of curves of genus $g$, has many applications in the study of principally polarized Abelian varieties, $\M_g$, and Brill--Noether theory. In particular, Welters defined in \cite{Welters_prym} the Prym--Brill--Noether loci \[V^r(f\colon\widetilde{C}\rightarrow C)\coloneqq\{L\in\Pic( \widetilde{C}) \;\mid\; \operatorname{Nm}_f(L)\cong\omega_C,~ h^0(\widetilde{C},L )\geq r+1\text{ and }h^0(\widetilde{C},L )\equiv r+1 \mod 2\}.\] It was subsequently shown in two papers \cite{Welters_prym,Bertram87} that when $g\geq \binom{r+1}{2}+1$, the locus $V^r(f\colon\widetilde{C}\rightarrow C)$ is non-empty of dimension at least $g-1-\binom{r+1}{2}$, and that equality is attained for generic $[f\colon\widetilde{C}\rightarrow C]\in\calR_g$. Moreover, when $g<\binom{r+1}{2}+1$, then $V^r(f\colon\widetilde{C}\rightarrow C)$ is empty for generic $[f\colon\widetilde{C}\rightarrow C]$. Recently, Schwarz investigated the Brill--Noether theory for general unramified cyclic covers of degree $n$, parameterized by $\calR_{g,n}$, and showed that for general $[f\colon\widetilde{C}\rightarrow C]\in\calR_{g,n}$, $\widetilde{C}$ admits no $\g{r}{d}$ if $\rho(g(\widetilde{C}),r,d)<-r$, where $g(\widetilde{C})=n(g-1)+1$ is the genus of $\widetilde{C}$, see \cite{SchwarzPrym} for more details.

In \Cref{sec: noncont from prym}, we consider the natural map \[\chi_g : \calR_g \to \M_{2g-1},~[f\colon\widetilde{C}\rightarrow C]\mapsto [\widetilde{C}],\] which sends the \'{e}tale double cover to the source curve, and investigate how the image, $\operatorname{Im}(\chi_g)$, interacts with the Brill--Noether stratification of $\M_{2g-1}$.

\section{Non-containments of pointed Brill--Noether loci of small codimension}\label{sec: non-cont for small rho}

The goal of this section is to provide some preliminary results that will be used to prove \Cref{theoremintro rho} via degeneration techniques. We want to find curves in the closure of $\mathcal{M}_{g,d}^r$ in $\Mbar_{g}$ that cannot be contained in the closure of another expected maximal Brill--Noether locus $\mathcal{M}_{g,e}^s$. As pointed Brill--Noether loci naturally appear in describing the boundary of Brill--Noether loci, in this section we will prove some non-containment results for them.

One key statement is that pointed Brill--Noether loci of expected codimension $1$ are not contained in pointed Brill--Noether loci of larger expected codimension.

\begin{prop} \label{prop divisors not contained in codim 2} Let $g,r,d,s,e$ be positive integers and let $a,b$ be vanishing sequences of type $(r,d)$ and respectively $(s,e)$, such that $\rho(g,r,d,a) = -1$ and $ \rho(g,s,e,b) \leq -2$. Then there is a non-containment 
\[ \mathcal{M}^r_{g,d}(a) \not \subseteq \mathcal{M}^s_{g,e}(b). \]
\end{prop}
\begin{proof}
 This result is an immediate consequence of \cite[Theorem 1.2]{EisenbudHarrisBN-1}. The locus $\mathcal{M}^r_{g,d}(a)$ is an irreducible divisor of $\mathcal{M}_{g,1}$ while the locus $\mathcal{M}^s_{g,e}(b)$ has codimension $2$ or higher.  
\end{proof}

This result can be extended to pointed Brill--Noether loci in $\mathcal{M}_{g,2}$.

\begin{cor} \label{cor divisors not contained in codim 2} Let $g,r,d,s,e$ be positive integers and let $a, b, c$ be vanishing sequences of type $(r,d)$ and respectively $(s,e)$, such that $\rho(g,r,d,a) = -1$ and $ \rho(g,s,e,b,c) \leq -2$. Then, letting $\pi\colon\mathcal{M}_{g,2} \rightarrow \mathcal{M}_{g,1}$ be the map forgetting the second marking, there is a non-containment 
\[ \pi^{-1}\mathcal{M}^r_{g,d}(a) \not \subseteq \mathcal{M}^s_{g,e}(b,c). \]	
\end{cor}
\begin{proof}
	Let $[\mathbb{P}^1, p, p_1,p_2] \in \mathcal{M}_{0,3}$ and consider the clutching map 
	\[ \mathcal{M}_{g,1} \rightarrow \overline{\mathcal{M}}_{g,2}\]
	sending a pointed curve $[C,q]$ to $[C\cup_{q\sim p} \mathbb{P}^1, p_1, p_2]$. The pullback of $\overline{\pi^{-1}\mathcal{M}^r_{g,d}(a)}$ along the clutching map is simply $\mathcal{M}^r_{g,d}(a)$, while the pullback of $\overline{\mathcal{M}}^s_{g,e}(b,c)$ consists of loci with Brill--Noether number strictly less than $-1$: For each limit $g^s_e$, $l\coloneqq (l_C, l_{\mathbb{P}^1})$ on a curve $[C\cup_{q\sim p} \mathbb{P}^1, p_1, p_2]$ we have the Brill-Noether additivity
    \[ -2 \geq \rho(g,s,e,b,c) \geq \rho(l_C,q) + \rho(l_{\mathbb{P}^1},p,p_1,p_2)   \]
    and the conclusion follows because the inequality $\rho(l_{\mathbb{P}^1},p,p_1,p_2) \geq 0$ is always satisfied, see \cite[Theorem 1.1]{EisenbudHarris-Kodaira>23}. \Cref{prop divisors not contained in codim 2} yields the conclusion.
\end{proof}

We want to show that containments are well-behaved with respect to the expected codimension, i.e., no Brill--Noether locus is contained in another Brill--Noether locus of higher expected codimension. We start with the case of codimension $2$. 

\begin{prop}\label{prop codim 2 not contained in codim 3}
	Let $\mathcal{M}^r_{g,d} \subseteq \mathcal{M}_g$ be a Brill--Noether locus satisfying $d < g+r, \ \rho(g,r,d) = -2$ and $r + 1 \leq g-d+r$. If $\mathcal{M}^s_{g,e}(b)$ is a pointed Brill--Noether locus with $\rho(g,s,e,b) \leq -3$, then letting $\pi\colon \mathcal{M}_{g,1}\rightarrow \mathcal{M}_g$ be the forgetful map, there is a non-containment 
	\[ \pi^{-1}\mathcal{M}^r_{g,d} \not \subseteq \mathcal{M}^s_{g,e}(b). \]
\end{prop}

\begin{proof}
	If $g \geq 4r+2$, we can consider a clutching map 
	\[ \mathcal{M}_{g_1,1}\times \mathcal{M}_{g_2,1} \rightarrow \Mbar_g \]
	with $g_1 = (r+1)k_1-1$ and $g_2 = (r+1)k_2-1$ for some $k_1, k_2 \geq 2$, where $k_1+k_2 = g-d+r$. 
	
	The locus $\mathcal{M}^r_{g_1,d}(rk_2-1, rk_2, \ldots, rk_2+r-1) \times \mathcal{M}^r_{g_2,d}(rk_1-1, rk_1, \ldots, rk_1+r-1)$ is a non-empty product of loci with Brill--Noether number $-1$, and appears in the pullback of $\Mbar^r_{g,d}$ via the clutching map as a result of \cite[Corollary 3.5]{basiclimitlinear}.

	We consider the diagram 
		\[
	\begin{tikzcd}
		\mathcal{M}_{g_1,1}\times \mathcal{M}_{g_2,2}  \arrow{r}{\iota}  \arrow{d}{} & \mathcal{M}_{g,1} \arrow{d}{\pi} \\
		\mathcal{M}_{g_1,1}\times \mathcal{M}_{g_2,1} \arrow{r}{}&  \mathcal{M}_g
	\end{tikzcd}
	\]
where the vertical maps are forgetful maps, while the horizontal maps are the obvious clutchings. 

By Brill--Noether additivity (cf. \cite[Proposition 4.6]{basiclimitlinear}) and \Cref{cor divisors not contained in codim 2}, the pullback of $\mathcal{M}^r_{g_1,d}(rk_2-1, rk_2, \ldots, rk_2+r-1) \times \mathcal{M}^r_{g_2,d}(rk_1-1, rk_1, \ldots, rk_1+r-1)$ to $\mathcal{M}_{g_1,1}\times \mathcal{M}_{g_2,2}$ is not contained in $\iota^{-1} \mathcal{M}^s_{g,e}(b)$. This implies the required non-containment 
\[ \pi^{-1}\mathcal{M}^r_{g,d} \not \subseteq \mathcal{M}^s_{g,e}(b). \]

We are left to treat the cases when $g < 4r+2$. In this situation, we have 
\[4r+4 > g+2=(r+1)(g-d+r) \geq (r+1)^2\]
and hence $1\leq r \leq 2$. 

If $r = 1$, then $2\leq g <6$, and the condition $\rho(g,r,d) = -2$ implies $g = 4$ and $d = 2$, whereby $\BN{g}{r}{d}$ is the hyperelliptic locus. 

Let $\mathcal{W}_2 \subseteq \mathcal{M}_{2,1}$ be the Weierstrass divisor and consider the clutching 
\[ \mathcal{M}_{2,1}\times \mathcal{M}_{2,1} \rightarrow \overline{\mathcal{M}}_{4}. \]
The locus $\mathcal{W}_2\times \mathcal{W}_2$ appears in the pullback of $\mathcal{M}^1_{4,2}$ via the clutching. The rest of the proof follows analogously to the case $g \geq 4r +2$. 

When $r = 2$, we have $7 \leq g < 10$, and the condition $\rho(g,2,d) = -2$ implies $g = 7$ and $d= 6$. We consider the clutching 
\[ \mathcal{M}_{2,1}\times \mathcal{M}_{5,1} \rightarrow \overline{\mathcal{M}}_{7}. \]
We take the product of codimension $1$ loci $\mathcal{M}^2_{2,6}(2,4,6) \times \mathcal{M}^2_{5,6}(0,2,4)$. By \cite[Corollary 3.5]{basiclimitlinear}, this locus appears in the pullback of $\overline{\mathcal{M}}^2_{7,6}$ via the clutching map. The proof of non-containment now follows as in the case $g \geq 4r+2$.
\end{proof}

In fact, the same argument as in the proof of \Cref{cor divisors not contained in codim 2} can be used to extend the result to codimension $2$ loci.

\begin{cor} \label{cor divisors not contained in codim 3} Let $g,r,d,s,e$ be positive integers and let $b, c$ be vanishing sequences of type $(s,e)$ such that $\rho(g,r,d) = -2$ and $ \rho(g,s,e,b,c) \leq -3$. Then, letting $\pi\colon\mathcal{M}_{g,2} \rightarrow \mathcal{M}_{g}$ be the map forgetting the markings, there is a non-containment 
\[ \pi^{-1}\mathcal{M}^r_{g,d} \not \subseteq \mathcal{M}^s_{g,e}(b,c). \]	
\end{cor}

This corollary, together with Brill--Noether additivity \cite[Proposition 4.6]{basiclimitlinear} will be the key results in proving \Cref{theoremintro rho}.

\section{Dimensionally expected non-containments}\label{sec: dim exp noncont}

In this section, we prove that given two expected maximal Brill--Noether loci $\BN{g}{r}{d}$ and $\BN{g}{s}{e}$ satisfying $\rho(g,r,d)>\rho(g,s,e)$ (i.e. the expected dimension of $\BN{g}{s}{e}$ is smaller than the expected dimension of $\BN{g}{r}{d}$), we have $\BN{g}{r}{d}\nsubseteq \BN{g}{s}{e}$. Our approach is in two steps. We first construct a chain curve $C_1 \cup C_2 \cdots \cup C_k$ appearing in the boundary of $\BN{g}{r}{d}$ by virtue of \cite[Corollary 3.5]{basiclimitlinear}. We then use Brill--Noether additivity to conclude that this curve does not admit a limit linear series of type $\g{s}{e}$, thus proving the non-containment $\BN{g}{r}{d}\nsubseteq \BN{g}{s}{e}$.

\begin{prop}\label{Prop BN loci split into BN divisors and codim 2}
    Let $\BN{g}{r}{d}$ be a Brill--Noether locus  satisfying the numerical condition
   \begin{itemize}
       \item[\normalfont{($\ast$)}] $\displaystyle (2r+1)\floor{\frac{-\rho(g,r,d)+1}2}-\floor{\frac{-\rho(g,r,d)}2} \leq g.$
   \end{itemize}
Then the closure of this locus in $\overline{\mathcal{M}}_{g}$ contains a chain curve $[C_1\cup C_2\cup\cdots C_k]$ such that
\begin{itemize}
    \item Each irreducible component $C_i$ is generic in a Brill--Noether locus $\BN{g_i}{r}{d_i}$ with 
    \[-1\geq \rho(g_i,r,d_i)\geq -2;\] 
    \item Each glueing point is generic on both irreducible components it connects.
\end{itemize}
\end{prop}

\begin{proof} Let $k = \floor{\frac{-\rho(g,r,d)+1}{2}}$ and consider the clutching 
    \[\varphi:\M_{g_1,1}\times \left(\prod_{i=2}^{k-1} \M_{g_i,2} \right)\times \M_{g_{k},1}\to \overline{\M}_g,\] 
    sending a tuple $\left( [C_1, p_1],[C_2, q_2^1, q_2^2],\dots, [C_{k-1}, q_{k-1}^1, q_{k-1}^2],[C_{k},p_{k}] \right)$ to the curve 
    \[\widetilde{C}\coloneqq C_1 \cup_{p_1\sim q_2^1} C_2 \cup_{q_2^2\sim q_3^1}C_3\cup\cdots\cup_{q_{k-1}^2\sim p_{k}} C_{k}.\] 
   We want to construct a chain curve $[C_1\cup C_2\cdots \cup C_k]$ admitting a smoothable limit $g^r_d$ and respecting the conditions in the hypothesis. We remark that it is sufficient to find a limit $g^r_d$ on this chain so that the vanishing orders are consecutive numbers for each node. Let $(v_1, v_1+1, \ldots, v_1+r)$ be the vanishing orders at $p_1$ for the $C_1$-aspect, $(v^1_i, v^1_i+1,\ldots, v^1_i+r)$ and $(v^2_i, v^2_i+1,\ldots, v^2_i+r)$ the vanishing orders at $q^1_i$ and $q^2_i$ for the $C_i$-aspect and $(v_k, v_k+1, \ldots, v_k+r)$ the vanishing orders at $p_k$ for the $C_k$-aspect. 
   
   \textbf{We treat first the case $\rho(g,r,d)$ is even.} 
    
    We show how to determine $g_i$ and $v_i^j$ from $g,r,d$. Note that $\rho(g,r,d) = -2k\equiv g \pmod{r+1}$. Starting with $(r-1,r-1,\dots,r-1)\in (\Z_{>0})^{\oplus k}$, we add $r+1$ to the first entry then the second, and so on, repeating cyclically until we obtain $(g_1, g_2\dots, g_{k})$ where $g_i\equiv -2\pmod{r+1}$ and $g=\sum_{i=1}^{k} g_i$. Let $v_i=\frac{g_i+2}{r+1}+d-r-g_i$. The vanishing orders are given inductively by 
    \begin{align*}
        v_1&=\frac{g_1+2}{r+1}+d-r-g_1\\
        v_2^1&=d-v_1-r\\
        v_2^2&=v_2-v_2^1\\
        v_i^1&=d-v_{i-1}^2-r\\
        v_i^2&=v_i-v_i^1=\frac{g_i+2}{r+1}-g_i+v^2_{i-1}.
    \end{align*}
    
    By construction, the vanishing orders satisfy the compatibility condition to be a refined limit linear series for $i\leq k-1$, and at $p_k$ we have 
    \begin{align*}
        v_k+r+v_{k-1}^2 &= \left( \sum\limits_{i=1}^{k} \frac{g_i+2}{r+1}-g_i \right) +2d-r\\
        &= \frac{g-\rho(g,r,d)}{r+1}-g+2d-r\\
        &=d,
    \end{align*} thus the compatibility condition is satisfied at every clutching point. Moreover, by definition $v_i=v_i^1+v_i^2$ and one checks that 
    \begin{align*}
        \rho(g,r,d,(v_1,\dots,v_1+r))&=-2,\\
        \rho(g,r,d,(v_i^1,\dots,v_i^1+r),(v_i^2,\dots,v_i^2+r))&=-2 ~\text{for $2\leq i \leq k-1$}, \text{ and}\\
        \rho(g,r,d,(v_k,\dots,v_k+r))&=-2.
    \end{align*}
    Finally, taking $d_i=d-v_i$, we note that the $i^{\text{th}}$ aspect corresponds to a $\g{r}{d_i}$ on $C_i$ which satisfies $\rho(g_i,r,d_i)=-2$, thus $\BN{g_i}{r}{d_i}$ is a Brill--Noether locus of codimension $2$.

    The locus of curves in $\operatorname{Im}(\varphi)$ admitting a $g^r_d$ with vanishing orders as above is of expected dimension and satisfies the conditions in the hypothesis. Finally, \cite[Corollary 3.5] {basiclimitlinear} implies that this locus appears in the closure of $\BN{g}{r}{d}$, as required.

    The condition $(\ast)$
    was tacitly used to ensure that $g_i > r-1$ for all $i$ and hence that $\BN{g_i}{r}{d_i}$ is non-empty, see \cite[Theorem 2.1]{bigas2023brillnoether}.
    
    \textbf{We now treat the case $\rho(g,r,d)$ is odd.} 

    We will keep the notations from the even case. In this situation, we have 
    \[\rho(g,r,d) =-2k+1 \equiv g \pmod{r+1}. \]

     Starting with $(r-1,r-1,\dots,r-1, r)\in (\Z_{>0})^{\oplus k}$, we add $r+1$ to the first entry then the second, and so on, repeating cyclically until we obtain $(g_1, g_2\dots, g_{k})$ where $g_i\equiv -2\pmod{r+1}$ for $1\leq i \leq k-1$, $g_k \equiv -1 \pmod{r+1}$ and $g=\sum_{i=1}^{k} g_i$. Let $v_i=\frac{g_i+2}{r+1}+d-r-g_i$ for $1\leq i \leq k-1$ and $v_k = \frac{g_k+1}{r+1}+d-r-g_k$. The vanishing orders are determined inductively by 
    \begin{align*}
        v_1&=\frac{g_1+2}{r+1}+d-r-g_1\\
        v_2^1&=d-v_1-r\\
        v_2^2&=v_2-v_2^1\\
        v_i^1&=d-v_{i-1}^2-r\\
        v_i^2&=v_i-v_i^1=\frac{g_i+2}{r+1}-g_i+v^2_{i-1}.
    \end{align*}
    By construction, a $g^r_d$ with these vanishing orders satisfies the compatibility condition to be a refined limit linear series for $i\leq k-1$, and at $p_k$ we have 
    \begin{align*}
        v_k+r+v_{k-1}^2 &= \left( \sum\limits_{i=1}^{k-1} \frac{g_i+2}{r+1}-g_i \right) + \frac{g_k+1}{r+1}-g_k +2d-r\\
        &= \frac{g-\rho(g,r,d)}{r+1}-g+2d-r\\
        &=d,
    \end{align*} thus the compatibility condition is satisfied at every clutching point. As before, $v_i=v_i^1+v_i^2$ by definition and one checks that $\rho(g_i,r,d-v_i)=-2$ for $1\leq i \leq k-1$ and $\rho(g_k,r,d-v_k)=-1$. Taking $d_i=d-v_i$ we obtain the Brill--Noether loci $\BN{g_i}{r}{d_i}$ having either codimension $1$ or $2$. By taking $[C_i] \in \BN{g_i}{r}{d_i}$ and glueing at generic points to form a chain $[C_1\cup C_2\cup \cdots \cup C_k]$ we obtain our desired curve.
\end{proof}

The numerical condition $(\ast)$ ensures that all the Brill--Noether loci we consider are non-empty. The condition is very mild. We identify precisely when the numerical condition above holds. 

\begin{lemma}\label{lemma when num cond holds}
    Let $\BN{g}{r}{d}$ be an expected maximal Brill--Noether locus. Then 
    \begin{itemize}
       \item[\normalfont{($\ast$)}] $\displaystyle(2r+1)\floor{\frac{-\rho(g,r,d)+1}2}-\floor{\frac{-\rho(g,r,d)}2} \leq g$
   \end{itemize} holds unless $\rho(g,r,d)=-(r+1)=-\ceil{\sqrt{g}}$ is odd and $g$ is not a square.
\end{lemma}
\begin{remark}
    We note that $(\ast)$ does not hold in general when $\rho(g,r,d)=-r-1$ is odd and $r=\ceil{\sqrt{g}-1}$, the expected maximal Brill--Noether locus $\BN{42}{6}{41}$ provides such an example. In fact, for any genus of the form $g=n^2-n$ with $\ceil{\sqrt{g}-1}$ even, the expected maximal Brill--Noether locus $\BN{g}{\floor{\sqrt{g}}}{d}$ contradicts $(\ast)$.
\end{remark}
\begin{proof}
    Assume that $\rho(g,r,d)$ is even, then \[(2r+1)\floor{\frac{-\rho(g,r,d)+1}2}-\floor{\frac{-\rho(g,r,d)}2} = -r\rho(g,r,d),\] and since for expected maximal loci $-\rho(g,r,d)\leq r+1$, we have $-r\rho(g,r,d)\leq r(r+1)$. To see that this holds for expected maximal Brill--Noether loci, first note that $r+1\leq g-d+r$. We now compute 
    \begin{align*}
        \rho(g,r,d)&\geq -r-1\\
        g+r+1 &\geq (r+1)(g-d+r) \geq (r+1)^2\\
        g&\geq r(r+1),
    \end{align*} as was to be shown.

    Assume now that $\rho(g,r,d)$ is odd. Then $(\ast)$ reads \[-r\rho(g,r,d)+r+1\leq g.\] As above, one sees that if $-\rho(g,r,d)\leq r-1$, then this holds. If $-\rho(g,r,d)=r$, then $(\ast)$ reads $r^2+r+1\leq g$, which clearly holds if $r\leq \sqrt{g}-1$. Similarly, if $-\rho(g,r,d)=r+1$, then $(\ast)$ reads $(r+1)^2\leq g$, which holds if $r\leq \sqrt{g}-1$. 
    
    Thus we may assume that $r=\ceil{\sqrt{g}-1}$ for $g$ not a square, and $r\leq -\rho(g,r,d)\leq r+1$. 

    It remains to show that $(\ast)$ holds when $-\rho(g,r,d)=r=\ceil{\sqrt{g}-1}=\floor{\sqrt{g}}$ is odd, $g$ is not a square, and $g\geq \floor{\sqrt{g}}^2+\floor{\sqrt{g}}$, see \Cref{rcases} in \Cref{subsec:background exp max BN loci}. In this case, $(\ast)$ reads \[r^2+r+1\leq g.\] We show that this holds. From \Cref{lem: exp max rho formula}, we have \[-\rho(g,r,d)=r\equiv r+1-(g\mod r+1),\] hence $g\equiv 1 \mod r+1$. Thus, as $g\geq \floor{\sqrt{g}}^2+\floor{\sqrt{g}}=r(r+1)$, we must have $g\geq r^2+r+1$, as claimed. 
\end{proof}

\begin{remark}\label{remark num cond not sat implies smallest rho}
    In particular, $(\ast)$ is satisfied for all but possibly one expected maximal Brill--Noether locus $\BN{g}{r}{d}$, the one with largest $r$ and smallest $\rho$. Indeed, when $(\ast)$ is not satisfied, then we have $\rho(g,r,d)<\rho(g,s,e)$ for all other expected maximal Brill--Noether loci $\BN{g}{s}{e}$.
\end{remark}

By imposing different requirements on the dimensions of the Brill--Noether loci $\BN{g_i}{r}{d_i}$ we can obtain similar results. The proof of \Cref{Prop BN loci split into BN divisors and codim 2} can be adapted to conclude these new results.
  \begin{prop}
    Let $\BN{g}{r}{d}$ be a Brill--Noether locus  satisfying the numerical condition
	\[-\rho(g,r,d)\cdot(2r+1)\leq g.\] 
 Then the closure of this locus in $\overline{\mathcal{M}}_{g}$ contains a chain curve $[C_1\cup C_2\cup\cdots C_k]$ such that
\begin{itemize}
    \item Each curve $C_i$ is generic in a Brill--Noether divisor $\BN{g_i}{r}{d_i}$ of some $\mathcal{M}_{g_i}$;
    \item Each glueing point is generic on both components it connects.
\end{itemize}
\end{prop} 

In fact, if we allow the ``clutching components" $\BN{g_i}{r}{d_i}$ of the expected maximal Brill--Noether loci to be of expected codimension $3$, a similar proposition holds with no numerical requirement.

\begin{prop}\label{Prop: exp max loci split into codim at most 3}
    Let $\BN{g}{r}{d}$ be an expected maximal Brill--Noether locus. The closure of this locus in $\overline{\mathcal{M}}_{g}$ contains a chain curve $[C_1\cup C_2\cup\cdots C_k]$ such that
\begin{itemize}
    \item Each curve $C_i$ is generic in a Brill--Noether locus $\BN{g_i}{r}{d_i}$ with $-1\geq \rho(g_i,r,d_i)\geq -3$; 
    \item Each glueing point is generic on both components it connects.
\end{itemize}
\end{prop} 

With these results in hand we prove our main theorem, that the dimensionally expected non-containments of expected maximal Brill--Noether loci hold.

\begin{theorem}\label{Thm:noncontainment based on rho}
    Let $\BN{g}{r}{d}$ and $\BN{g}{s}{e}$ be expected maximal Brill--Noether loci. If $\rho(g,s,e)<\rho(g,r,d)$, then $\BN{g}{r}{d} \nsubseteq \BN{g}{s}{e}$.
\end{theorem}
\begin{proof}
    As noted in \Cref{remark num cond not sat implies smallest rho}, the condition $(\ast)$ of \Cref{Prop BN loci split into BN divisors and codim 2} holds unless \[-\rho(g,r,d)=r+1=\ceil{\sqrt{g}-1}+1\] is odd and $g$ is not a square, whereby $\rho(g,r,d)<\rho(g,s,e)$ for all expected maximal loci $\BN{g}{s}{e}$. By assumption, we have $\rho(g,s,e)<\rho(g,r,d)$, thus we may assume $(\ast)$ holds.
    
    Consider a chain curve
        \[\widetilde{C}\coloneqq C_1 \cup_{p_1\sim q_2^1} C_2 \cup_{q_2^2\sim q_3^1}C_3\cup\cdots\cup_{q_{k-1}^2\sim p_{k}} C_{k}\]
    in the boundary of $\BN{g}{r}{d}$ as described in \Cref{Prop BN loci split into BN divisors and codim 2}. Each irreducible component $C_i$ is generic in a Brill--Noether locus of codimension $1$ or $2$, depending on the parity of $\rho(g,r,d)$ as in \Cref{Prop BN loci split into BN divisors and codim 2}. 
    
    Assume for contradiction that we have the containment $\BN{g}{r}{d}\subseteq \BN{g}{s}{e}$. This implies that $\widetilde{C}$ admits a limit $\g{s}{e}$. Denoting the aspects of the limit $\g{s}{e}$ by $l_i$, \Cref{prop divisors not contained in codim 2}, \Cref{prop codim 2 not contained in codim 3} and \Cref{cor divisors not contained in codim 3} imply that $\rho(l_1,p_1)\geq -2$, $\rho(l_i,q_i^1, q_i^2)\geq -2$, and $\rho(l_{k},p_{k})\geq\begin{cases}
        -1 & \text{if $\rho(g,r,d)$ is odd}\\
        -2 & \text{if $\rho(g,r,d)$ is even}
    \end{cases}$. Brill--Noether additivity gives \[ \rho(g,s,e)\geq \rho(l_1,p_1)+\left( \sum_{i=2}^{k-1} \rho(l_i,q_i^1, q_i^2)\right)+\rho(l_{k},p_{k})\geq -2k+2 +\begin{cases}
        -1 & \text{if $\rho(g,r,d)$ is odd}\\
        -2 & \text{if $\rho(g,r,d)$ is even}
    \end{cases} = \rho(g,r,d),\] contradicting $\rho(g,s,e)<\rho(g,r,d)$.
\end{proof}

\section{Existence of components of expected dimension} \label{sec: comp of exp dim}

The question of whether Brill--Noether loci, or more generally the schemes $\G^r_{g,d}$, have components of the expected dimension has recently received attention in the work of many authors, in particular Pflueger and Teixidor i Bigas \cite{pflueger_legos,bigas2023brillnoether}. They show that when $-\rho(g,r,d)\leq g-3$, then there exists a component of the expected dimension (or expected relative dimension for $\G^r_{g,d}\to\M_g$) \cite[Theorem A]{pflueger_legos}, and in case $d\neq g-1$, then this also holds for $-\rho(g,r,d)\leq g-2$ \cite[Theorem 2.1]{bigas2023brillnoether}. We give a new proof of the existence of components of expected dimension in a smaller range.

Reasoning as in \Cref{Prop BN loci split into BN divisors and codim 2} immediately gives components of the expected dimension.

\begin{theorem}\label{Thm: comp of exp dimension}
         If $d\leq 2g-2$ and $-\rho(g,r,d)\leq \ceil{g/2}$, then $\BN{g}{r}{d}$ has a component of the expected dimension.   
\end{theorem}

\begin{proof} Using Serre duality, we can assume $d\leq g-1$. The low genus cases $2\leq g \leq 7$ are an immediate consequence of \cite{bigas2023brillnoether}, while the case $r=1$ is well-known in the literature, see \cite{Farkas_2001} and \cite{Arbarello-Cornalba}. We assume $g\geq 8, r\geq 2$ and prove the statement by reasoning inductively. We will consider two cases, depending on how large the value $-\rho(g,r,d)$ is.
    
    \textbf{Case I:} We assume $-\rho(g,r,d) \geq r$. 
    
    In this case, we consider a (hyperelliptic) curve $[C_1] \in \BN{r+2}{r}{2r}$ and a curve $[C_2]\in \BN{g-r-2}{r}{d-r}$ and let $p_1\in C_1$ and $p_2\in C_2$.

    We know that the locus $\BN{r+2}{r}{2r} = \BN{r+2}{1}{2}$ is irreducible of codimension $r$. By induction, we also know that $\BN{g-r-2}{r}{d-r}$ has a component of expected dimension, as the numerical conditions in the hypothesis are satisfied:  
    \begin{itemize} 
        \item The condition \[ \rho(g-r-2, r, d-r) = \rho(g,r,d) + r \geq -\ceil{\frac{g-r-2}{2}} \]
    is an immediate consequence of $r\geq 2$ and the hypothesis $\rho(g,r,d)\geq -\ceil{g/2}$.
       \item For the condition $d-r \leq 2(g-r-2)-2$, i.e. $d\leq 2g-r-6$, we use $d\leq g-1$. If the condition is not satisfied, we obtain the inequality 
    \[ 2g-r-5 \leq d \leq g-1\] 
    and hence $g \leq r+4$ and $d\leq r+3$. Clifford's inequality $2r\leq d$ implies $r\leq 3$ and hence $g\leq 7$, contradicting our assumption.
    \end{itemize}
    
    By taking $[C_2]$ in a component of expected dimension of $\BN{g-r-2}{r}{d-r}$ and reasoning as in the proof of \Cref{Prop BN loci split into BN divisors and codim 2} we obtain that $[C_1\cup_{p_1\sim p_2}C_2] \in \overline{\mathcal{M}}^r_{g,d}$. In particular, we found a locus having expected codimension in the boundary of $\overline{\mathcal{M}}_g$. This locus must be contained in a component of $\BN{g}{r}{d}$ of expected codimension $-\rho(g,r,d)$.
    
    \textbf{Case II:} Assume that $-\rho(g,r,d) \leq r-1$. 

   In this situation, we consider \[[C_1]\in \BN{3r+3+\rho(g,r,d)}{r}{4r+\rho(g,r,d)} \text{ and } [C_2] \in \BN{g-3r-3-\rho(g,r,d)}{r}{d-3r-\rho(g,r,d)}.\]
    We note that the genus $g-3r-3-\rho(g,r,d)$ is nonnegative. Indeed, from \Cref{lem: small -rho means exp max}, we see that $\BN{g}{r}{d}$ is expected maximal,  hence 
   \begin{equation*}
        r \leq 
        \begin{cases}
	       \ceil{\sqrt{g}-1} & \text{if~} g\geq\floor{\sqrt{g}}^2+\floor{\sqrt{g}}\\
	       \floor{\sqrt{g}-1} & \text{if~} g<\floor{\sqrt{g}}^2+\floor{\sqrt{g}},
        \end{cases}
   \end{equation*}
   and we note that the inequality 
   \[ g\leq 3r+2 + \rho(g,r,d) \]
   cannot be satisfied for $g\geq 8$. We also note that since $\BN{g}{r}{d}$ is expected maximal and $g\geq 8$, the degree $d-3r-\rho(g,r,d)$ is non-negative.
   
   Reasoning inductively we see that $\BN{3r+3+\rho(g,r,d)}{r}{4r+\rho(g,r,d)}$ has a component of codimension $-\rho(g,r,d)$ in $\mathcal{M}_{3r+3+\rho(g,r,d)}$. 
   
   Moreover, as 
   \[ \rho(g-3r-3-\rho(g,r,d), r, d-3r-\rho(g,r,d)) = 0, \]
    we obtain $\BN{g-3r-3-\rho(g,r,d)}{r}{d-3r-\rho(g,r,d)} = \mathcal{M}_{g-3r-3-\rho(g,r,d)}$, hence the Brill--Noether locus has codimension $0$, and has a component of expected dimension.

   Reasoning as in the proof of \Cref{Prop BN loci split into BN divisors and codim 2} we get that $[C_1\cup_{p_1\sim p_2} C_2] \in \overline{\mathcal{M}}^r_{g,d}$ when $[C_1]$ is contained in a component of expected dimension of $\BN{g-3r-3-\rho(g,r,d)}{r}{d-3r-\rho(g,r,d)}$. 

   In particular, we found a locus having expected codimension $-\rho(g,r,d)$ in the boundary. This locus must be in the intersection of the boundary with a component of $\BN{g}{r}{d}$ having codimension $-\rho(g,r,d)$ in $\mathcal{M}_g$.
   \end{proof}

\section{Non-containments obtained from Prym}\label{sec: noncont from prym}
 
In this section, we look at the Prym moduli space $\mathcal{R}_g$ parametrizing unramified double covers $[f\colon\widetilde{C}\rightarrow C]$ of genus $g$ curves, and consider the map 
 	\[ \chi_g\colon \mathcal{R}_g \rightarrow \mathcal{M}_{2g-1} \]
 sending the double cover $[f\colon\widetilde{C}\rightarrow C]$ to the source curve $\widetilde{C}$. In analogy to \cite{ahl_BN_via_gonality_2023}, where gonality loci were used to distinguish Brill--Noether loci, we consider how $\operatorname{Im}(\chi_g)$ intersects the Brill--Noether stratification of $\M_{2g-1}$, thereby obtaining new non-containments of Brill--Noether loci.
 
 The following proposition is an immediate consequence of \cite[Theorem 1.4]{Bertram87}.
\begin{prop}\label{prop: prym containment}
	Let $g = 1 + \frac{r(r+1)}{2}+\varepsilon$ for $0\leq \varepsilon < \frac{r}{2}$. Then \[\operatorname{Im}(\chi_g) \subseteq \mathcal{M}^r_{\widetilde{g}, 2g-2}.\]
 where $ \widetilde{g} = 2g-1 =  1 + r(r+1) + 2\varepsilon$.
\end{prop}

\begin{proof} We have the following obvious containment between Prym--Brill--Noether and Brill--Noether spaces: 
\[ V^r(f\colon\widetilde{C}\rightarrow C) \subseteq W^r_{2g-2}(\widetilde{C}) \]
By \cite[Theorem 1.4]{Bertram87}, $V^r(f\colon\widetilde{C}\rightarrow C) \neq \emptyset$ for any $[f\colon\widetilde{C}\rightarrow C] \in \mathcal{R}_g$, it follows that any $\widetilde{C}$ in the image of $\chi_g$ admits a $g^r_{2g-2}$, i.e. \[\operatorname{Im}(\chi_g) \subseteq \mathcal{M}^r_{2g-1,2g-2}.\qedhere\]
\end{proof}

We remark that $\BN{2g-1}{r}{2g-2}$ is expected maximal. Indeed, as $\tilde{g}=2g-1$, we have \[-r-1\leq \rho(2g-1,r,2g-2) = 2g-1 - (r+1)(r+1) = 2\varepsilon -r  \leq -1 \]
and hence as $r\leq \sqrt{2g-1}$, we see that $r$ satisfies \Cref{rcases} (with genus $\tilde{g}=2g-1$), hence \Cref{lem: small -rho means exp max} shows that $\BN{2g-1}{r}{2g-2}$ is expected maximal.

Conversely, \cite[Theorem 1.1]{SchwarzPrym} shows that $\operatorname{Im}(\chi_g)$ is not contained in certain Brill--Noether loci.

\begin{prop}\label{prop: prym non cont schwarz}
    Let $\widetilde{g} = 2g-1$ and $r, d$ two numbers such that $\rho(\widetilde{g},r,d) = -r-1$. Then we have the non-containment 
    \[\operatorname{Im}(\chi_g) \not\subseteq \mathcal{M}^r_{\widetilde{g},d}.\]
\end{prop}  

Using the method of \cite[Theorem 0.4]{FarkasAprodu-Greenconj} we can prove that $\operatorname{Im}(\chi_g)$ is not contained in certain Brill--Noether loci.

\begin{prop}\label{prop: prym non cont limit lin}
    Let $\widetilde{g} = 2g-1$ and $r, d$ two numbers such that $\rho(\widetilde{g},r,d) = -r$ and either 
	\begin{itemize}
		\item  $r$ is even and $d$ is odd, or 
		\item  $r\equiv 1 \pmod{4}$ and $d$ is odd.
	\end{itemize}
Then we have the non-containment \[\operatorname{Im}(\chi_g) \not\subseteq \mathcal{M}^r_{\widetilde{g},d}.\]
\end{prop}

\begin{proof}
	We assume $\operatorname{Im}(\chi_g) \subseteq \mathcal{M}^r_{\widetilde{g},d}$ and we will reach a contradiction. For this, we will provide a curve in the closure $\overline{\operatorname{Im}(\chi_g)}$ that does not admit a limit $g^r_{d}$. 
	
   As in the proof of \cite[Theorem 0.4]{FarkasAprodu-Greenconj}, let $\pi_E\colon \widetilde{E} \rightarrow E$ be an \'etale double cover of an elliptic curve, $p \in E$ and $\left\{x,y\right\} \coloneqq \pi_E^{-1}(p)$. Taking $[C_1,p_1]$ and $[C_2, p_2]$, two copies of a generic pointed curve $[C,p] \in \mathcal{M}_{g-1,1}$, we obtain a double cover 
	\[ [C_1\cup_{p_1\sim x}\widetilde{E}\cup_{y\sim p_2}C_2\rightarrow C\cup_p E] \in \overline{\mathcal{R}}_g \]
	and see that $ \widetilde{C}\coloneqq[C_1\cup\widetilde{E}\cup C_2/p_1\sim x, p_2\sim y] \in \overline{\operatorname{Im}(\chi_g)}$, see the boundary description of $\overline{\mathcal{R}}_g$ in \cite{FarkasPrym} and \cite{Casa}.
	Assume that $\widetilde{C}$ admits a limit $g^r_d$ and denote by $l_1, \widetilde{l}$ and $l_2$ its aspects over the curves $C_1, \widetilde{E}$ and $C_2$. Moreover, we denote by $w_i$ the vanishing orders of $l_i$ at the node $p_i$ for $i =1,2$ and by $\widetilde{w}_1, \widetilde{w}_2$ the vanishing orders of $\widetilde{l}$ at the points $x$ and $y$. 
	
	By Brill--Noether additivity, we have 
	\[ \rho(2g-1,r,d) = -r \geq \rho(l_1, p_1) + \rho(l_2,p_2) + \rho(\widetilde{l}, x, y) \geq 0 + 0 + (-r) = -r. \]
	We have used here that the Brill--Noether number is non-negative for every  linear series on a generic pointed curve $[C,p] \in \mathcal{M}_{g-1,1}$, see \cite[Theorem 1.1]{EisenbudHarris-Kodaira>23}, and that $\rho(\widetilde{l}, x, y) \geq -r$ for every $g^r_d$ and every two points on an elliptic curve, see \cite[Proposition 1.4.1]{FarkasThesis}. 
	
	This double inequality implies that $\rho(l_1,p_1) = \rho(l_2,p_2) = 0$ and $\rho(\widetilde{l}, x, y) = -r$ and the limit linear series is refined. Let $(a_0, \ldots, a_r)$ and $(b_0,\ldots, b_r)$ be the entries of $\widetilde{w}_1$ and $\widetilde{w}_2$, respectively.
	
	Because $\rho(\widetilde{l}, x, y) = -r$, we must have $a_i + b_{r-i} = d$ for every $0\leq i \leq r$. Moreover, because $2x\equiv 2y$ all the $a_i$'s have the same parity. Implicitly, all the $b_i$'s have the same parity. 
	
	Because the limit linear series is refined (i.e. Brill-Noether additivity gives an equality), we must have $w_2 = (a_0,\ldots, a_r)$ and $w_1 = (b_0,\ldots, b_r)$. 
	
	Because $\rho(g-1,r,d,w_1) = \rho(g-1,r,d, w_2) = 0$ we get that 
	\[ \sum_{i=0}^{r}a_i = \sum_{i=0}^r b_i = \frac{(r+1)d}{2}. \]
   
  When $r$ is even and $d$ is odd, this is impossible.
	
    When $r\equiv 1 \pmod{4}$ and $d$ is odd, we obtain the contradiction 
    \[ 0 \equiv \sum_{i=0}^r a_i \equiv \frac{(r+1)d}{2} \equiv 1 \pmod{2}.\] Therefore the curve  $ \widetilde{C}$ does not admit any limit $g^r_d$.
\end{proof}

As a consequence of \Cref{prop: prym non cont schwarz} and \Cref{prop: prym non cont limit lin}, we obtain new non-containments of Brill--Noether loci.
\begin{cor}\label{cor noncontainments from prym} 
    Let $g = 1 + r(r+1) + 2\varepsilon$ for some $0\leq \varepsilon < \frac{r}{2}$ and let $s,d$ be positive integers satisfying either 
    \begin{itemize}
        \item $\rho(g,s,d) = -s-1$, or 
        \item $\rho(g,s,d) = -s$, $d$ is odd and $s \not\equiv 3 \pmod{4}$.
    \end{itemize}
    Then there is a non-containment 
    \[ \mathcal{M}^r_{g,g-1} \nsubseteq \mathcal{M}^s_{g,d}. \]
\end{cor}
\begin{proof}
    Let $g' \coloneqq 1 + \frac{r(r+1)}{2} + \varepsilon$. By \Cref{prop: prym containment}, a generic element in the locus $\operatorname{Im}(\chi_{g'})$ is contained in $\mathcal{M}^r_{g,g-1}$ but \Cref{prop: prym non cont schwarz} or \Cref{prop: prym non cont limit lin} show that  $\operatorname{Im}(\chi_{g'})\nsubseteq \mathcal{M}^s_{g,d}$. The conclusion follows. 
\end{proof}

This gives infinitely many non-containments of expected maximal Brill--Noether loci of the form $\BN{g}{r}{d}\nsubseteq \BN{g}{s}{e}$ with $s<r$, which has been heretofore out of reach of other techniques in general. We give an example of an infinite family of non-containments by taking $\varepsilon=0$.

\begin{cor}\label{cor: eps=0 noncontainments}
    Let $r$ be an even integer not divisible by $4$ and let $g=r^2+r+1$. Then we have a non-containment of expected maximal Brill--Noether loci \[\BN{g}{r}{g-1}\nsubseteq \BN{g}{r-1}{g-3}.\]
\end{cor}
\begin{proof}
    One checks that $\rho(g,r,g-1)=-r$, and $\rho(g,r-1,g-3)=-r+1$. The result follows from \Cref{cor noncontainments from prym}.
\end{proof}

By taking larger values of $\varepsilon$, one might potentially obtain further families of non-containments of expected maximal Brill--Noether loci.

\begin{remark}
    These results, however, cannot show the conjectured non-containments of expected maximal Brill--Noether loci of the form \[\BN{r^2+r}{r}{r^2+r-1}\nsubseteq \BN{r^2+r}{r-1}{r^2+r-3}.\] In fact, at present, these non-containments remain out of reach in general for all known techniques.
\end{remark}

\vfill
\end{document}